\documentclass[reqno]{amsart}
\usepackage{amsmath}%
\usepackage{amsfonts}%
\usepackage{amssymb}%
\usepackage{graphicx}
\usepackage{mathrsfs}
\usepackage{hyperref}
\usepackage{fullpage}
\usepackage{color}
\usepackage{cite}
\def\q{\hfill\rule{1ex}{1ex}}
\def\0{\emptyset}

\def\q{\hfill\rule{1ex}{1ex}}

\newtheorem{theorem}{Theorem}
\newtheorem{lemma}[theorem]{Lemma}

\newtheorem{claim}[theorem]{Claim}

\title
  {The existence of Hamilton cycle in $n$-balanced $k$-partite graphs}

\thanks{
Yi Zhang is supported by the National Natural Science Foundation of China (Grant 11901048 \& 12071002).}

\author{Zongyuan Yang}
\address{ School of Science, Beijing University of Posts and Telecommunications, Beijing, 100876, China}
\email{yangzongyuan0@bupt.edu.cn}

\author{Yi Zhang}
\address{ School of Science, Beijing University of Posts and Telecommunications, Beijing, 100876, China}
\email{shouwangmm@sina.com}

\author{Shichang Zhao}
\address{ School of Science, Beijing University of Posts and Telecommunications, Beijing, 100876, China}
\email{1051399257@qq.com}

\date{}

\begin{document}
\title{The existence of Hamilton cycle in $n$-balanced $k$-partite graphs}

\begin{abstract}
 Let $G_{k,n}$ be the $n$-balanced $k$-partite graph, whose vertex set can be partitioned into $k$ parts, each has $n$ vertices. In this paper, we prove that if $k \geq 2,n \geq 1$, for the edge set $E(G)$ of $G_{k,n}$
 $$|E(G)| \geq\left\{\begin{array}{cc}
   1                         & \text { if } k=2, n=1 \\
   n^{2} C_{k}^{2}-(k-1) n+2 & \text { other }
  \end{array}\right.$$
 then $G_{k,n}$ is hamiltonian. And the result may be the best.
\end{abstract}

\maketitle

\section{1 \quad Introduction}

\quad A Hamiltonian circle is a circle that contains all the vertices of the graph. The problem of the existence of Hamiltonian circle has been a highly  significant problem in graph theory, because it's NP-hard. Researchers have tried to give some tight sufficient conditions to prove the existence of Hamiltonian circles, among which the minimum degree condition (Dirac condition), the minimum degree sum condition (Ore condition) and the Fan-type condition are three classical results.

Inspired by the problem of Turán-type extremal graphs, we hope to give a sufficient condition for the existence of Hamiltonian circles in balanced multipartite graphs in terms of the number of edges, i.e., there must be a Hamiltonian circle for as many edges as there are at least in a balanced multipartite graph. Complete balanced multipartite graphs have a lot of great structural properties, such as having Hamiltonian circles, which can be used as a structural basis for interconnection networks. Connections are often broken in networks, so the fault tolerance of edges can be an important indicator of the stability of a network. The paper can provide an important theoretical support for determining the fault tolerance of complete multipartite graphs with Hamiltonian circles from the perspective of the number of edges, making it the basis of network structure.

 Let $G_{k,n}$ be the complete $n$-balanced $k$-partite graph, whose vertex set can be partitioned into $k$ parts, each has $n$ vertices. Our main result is the following theorem.\\
\begin{theorem}\label{theorem1} Let \normalsize $G=(V,E)$ be an $n$-balanced $k$-partite graph with $k \geq 2$,  $n \geq 1$ except $k=2$, $n=1$.  If $|E(G)| \geq n^{2} C_{k}^{2}-(k-1) n+2 $, then $G$ is Hamiltonian.
\end{theorem}
\section{2  \quad Notations and useful results}

\quad All graphs considered in this paper are simple, finite, loop-free and undirected. For a graph $G$, let $V(G),E(G)$ denote the vertex set and edge set of $G$, respectively. For a vertex $u \in V(G)$, let $N_G(u)$ be the set of adjacent vertices to $u$ in $G$ and $d_G(u) =|N_G(u)|$.

Let $G_{n_1,\cdots,n_k}=(V(G_{n_1,\cdots,n_k}),E(G_{n_1,\cdots,n_k}))$ be a $k$-partite graph with order $n_1+n_2+\cdots+n_k$, where the vertex set $V(G_{n_1,\cdots,n_k})$ can be divided into $k$ parts $V_1,V_2,\cdots,V_k$ with $|V_i|=n_i$ for $1 \leq i \leq k$ and the edge set $E(G_{n_1,\cdots,n_k})$ contains edges with no two vertices from one part.   If the edge set $E(G_{n_1,\cdots,n_k})$ contains all edges with no two vertices from one part, then we call $G_{n_1,\cdots,n_k}$ the complete $k$-partite graph with order $n_1+n_2+\cdots+n_k$, and denote the graph by $CG_{n_1,\cdots,n_k}$.  If $n_1=\cdots=n_k$, for simplicity, we write $G_{k,n}$ and $CG_{k,n}$ instead, and call $G_{k,n}$ an $n$-balanced $k$-partite graph. Let $\overline{G_{n_1,\cdots,n_k}}$ be the $k$-partite graph with the vertex set $V(G_{n_1,\cdots,n_k})$ and the edge set
$E(CG_{n_1,\cdots,n_k})\setminus E(G_{n_1,\cdots,n_k})$.  Clearly $|E(G_{n_1,\cdots,n_k})|+|E(\overline{G_{n_1,\cdots,n_k}})| = |E(CG_{n_1,\cdots,n_k})|$ and $|E(G_{k,n})|+|E(\overline{G_{k,n}})| = |E(CG_{k,n})|$.
Given a vertex $u\in V(G_{k,n})$,  clearly $d_{\overline{G_{k,n}}}(u)+d_{G_{k,n}}(u) =(k-1)n$. Let $ \delta(G_{k,n})$ be the smallest degree among all vertices and $ \sigma(G_{k,n})$ be the smallest degree sum of two nonadjacent vertices from different parts.

The following results will be used in our proof.
\begin{theorem}(Ore \cite{1960Note}) Let $G$ be a graph with  $n \geq 3$ vertices. If $d(u)+d(v) \geq n$ for any two nonadjacent vertices $u$ and $v$, then $G$ is hamiltonian.
\end{theorem}

\begin{theorem}\label{theorem2} For any graph $G$ with $n \geq 3$ vertices, if $|E(G)| \geq \binom{n-1}{2} + 2$,
then $G$ is Hamiltonian.
\end{theorem}
{\noindent Proof.} We claim that $d(u)+d(v) \geq n$ for any two nonadjacent vertices $u$ and $v$.  Otherwise, we let $u_0$ and $v_0$ be two nonadjacent vertices with $d(u)+d(v) \leq n-1$. Then we obtain that  $|E(G)| \leq \binom{n-2}{2} + n-1 < \binom{n-1}{2} + 2$, a contradiction.  Furthermore, by Ore Theorem,  $G$ is Hamiltonian.

\begin{lemma}\label{lemma1} Let $G$ be a graph with $n \geq 3$ vertices and $u_1-P-u_n$ be a Hamilton path. If $d(u_1)+d(u_n) \geq n$, then $G$ is Hamiltonian.
\end{lemma}
{\noindent Proof.} Let $u_1-u_2-\cdots u_{n-1}-u_n$ be a Hamilton path. To the contrary, we assume that $G$ is not Hamiltonian. If $u_k$ ia adjacent to $u_1$, then $u_{k-1}$ is not adjacent to $u_n$, otherwise we find a Hamilton cycle. Therefore $d(u_n) \leq n-1-d(u_1)$. Similarly,  $d(u_1) \leq n-1-d(u_n)$. It follows that $d(u_1)+d(u_n) \leq  n-1$, a contradiction. \q

\begin{theorem}\label{theorem4} \cite{1995Hamiltonicity} Let $G_{k,n}$ be an $n$-balanced $k$-partite graph with $k \geq 2$. If $$\delta(G_{k,n})>\left\{\begin{array}{ll}
  \left(\frac{k}{2}-\frac{1}{k+1}\right) n & \text { if } k \text { is odd, }  \\
  \left(\frac{k}{2}-\frac{2}{k+2}\right) n & \text { if } k \text { is even, }
 \end{array}\right.$$ then $G_{k,n}$ is Hamiltonian.
\end{theorem}

\begin{theorem}\label{theorem3} \cite{1997Degree} Let $G_{k,n}$ be an $n$-balanced $k$-partite graph with $k \geq 2$. If $$\sigma(G_{k,n})>\left\{\begin{array}{ll}
  \left(k-\frac{2}{k+1}\right) n & \text { if } k \text { is odd, }  \\
  \left(k-\frac{4}{k+2}\right) n & \text { if } k \text { is even, }
 \end{array}\right.$$ then $G_{k,n}$ is Hamiltonian.
\end{theorem}
\section{ Proof of Theorem \ref{theorem1}}

Suppose that  $G_{k,n}=(V(G_{k,n}),E(G_{k,n}))$ is an $n$-balanced $k$-partite graph with $k$ parts  $V_1$, $V_2$, $\cdots,$ $V_k$ and  $|E(G_{k,n})| \geq \binom{k}{2}n^{2}-(k-1) n+2 $, where $k \geq 2$,  $n \geq 1$ except $k=2$, $n=1$. Recall that $E(\overline{G_{k,n}})=E(CG_{k,n})\setminus E(G_{k,n})$,  we have
\begin{equation}\label{eq111}
 |E(\overline{G_{k,n}})| \leq(k-1) n-2.
\end{equation}
If $n=1$, then $|E(G_{k,1})| \geq \binom{k-1}{2}+2$, by Theorem \ref{theorem2}, the result holds.

\begin{claim} If $k=2$, the result holds.
\end{claim}
\begin{proof} In this case, we have $|E(G_{2,n})| \geq n^2-n+2$ as $k=2$, which implies that
\begin{equation}\label{eq222}
\sigma {G_{2,n}} \geq n+1 \text{\, and\, } \delta (G_{2,n}) \geq 2.
\end{equation}
Otherwise $|E(G_{2,n})| \leq  n^2-n+1$, a contradiction.
By Theorem \ref{theorem3} with $k=2$, the result holds.
\end{proof}
\begin{claim}\label{claim88} If $n=2$, the result holds.
\end{claim}
\begin{proof} In this case we have $|E(G_{k,2})| \geq 2k^2-4k+4$ as $n=2$, which implies that
 \begin{equation}\label{eq2}
 \sigma (G_{k,2}) \geq 2k-1 \text{\, and\, } \delta (G_{k,2}) \geq 2.
\end{equation}
Otherwise $|E(G_{k,2})| \leq  2k^2-4k+3$, a contradiction.

Theorem \ref{theorem3} with $n=2$ implies that if
\begin{equation}\label{eq1}
 \sigma (G_{k,2}) >\left\{\begin{array}{ll}
   2k-\frac{4}{k+1} & \text { when } k \text { is odd, }  \\
   2k-\frac{8}{k+2}  & \text { when } k \text { is even, }
 \end{array}\right.
\end{equation}
then $G_{k,2}$ is Hamiltonian.
It follows directly that the result holds for $k=2,4$.

If $\sigma (G_{k,2}) \geq 2k$, then  $ G_{k,2}$ is Hamiltonian as (\ref{eq1}). By (\ref{eq2}),  we just need to consider the case when $ \sigma (G_{k,2}) =2k-1$.  Without loss of generality, we let  $ u_1 \in V_1$ and $ u_2 \in V_{2}$ be two nonadjacent vertices  such that $ d(u_1)+d(u_2)=2k-1$ and $d(u_2) > d(u_1)$. It follows that $d(u_2) \geq k$. Since $|E(G_{k,2})| \geq 2k^2-4k+4$, we obtain that $G_{k,2}-\{u_1,u_2\}$ is the complete $k$-partite  graph with order $1+1+2+\cdots+2=2k-2$. Otherwise $|E(G_{k,2})| \leq 2k^2-4k+3$, a contradiction. Suppose $k=3$. We obtain that $d(u_1)=2$, $d(u_2)=3$.  Combining with  $|E(G_{3,2})| \geq  10$, it is easy to find a Hamilton cycle in $G_{3,2}$.

We prove the case when $k \geq 5$ by induction on $k$.  It is easy to obtain that  $G_{k,2}[V_2\cup V_3\cup\cdots V_k]$ is a 2-balanced $k-1$ partite graph with at least $\binom{k-2}{2}2^2+2(k-2)+k-1  > 2(k-1)^2-4(k-1)+4 $ as $k \geq 5$. By inductive hypothesis, $G_{k,2}[V_2\cup V_3\cup\cdots V_k]$ contains a Hamilton cycle $C$. Let $u_1'$ be the other vertex of $V_1$ different from $u_1$.   By (\ref{eq2}),  $\delta (G_{k,2}) \geq 2$,  we can let $u, v \in V(C)$ be two vertices adjacent to $u_1$.  If $\{u,v\}\in E(C)$, then we can construct a cycle of size $2k-1$ in $G_{k,2}-u_1'$. Since $d_{G_{k,2}}(u_1') \geq 2$, it is easy to find a Hamilton path in $G_{k,2}$ with one end being $u_1'$ and the other not being $u_1$, denoted by $w$. Clearly $d(u_1')+d(w) \geq 2(k-2)+1+k > 2k$ as $k \geq 5$. By Lemma \ref{lemma1}, $G_{k,2}$ has a Hamilton cycle. If $\{u,v\}\notin E(C)$, suppose that $C=u-u'-C_1-v-v'-C_2-u$, where $u'-C_1-v$ is the path from $u'$ to $v$ on $C$ not passing $u$ and  $v'-C_2-u$ is the path from $v'$ to $u$ on $C$ not passing $v$,   then at least one of $u'$ and $v'$ is adjacent to $u_1'$ as $d(u_1') \geq 2(k-2)+1$, say $\{u_1',v'\}\in E(G_{k,2})$. Then $u_1'-v'-C_2-u-u_1-v-C_1-u'$ is a Hamilton path in $G_{k,n}$.  Clearly $d(u_1')+d(u') \geq 2(k-2)+1+k > 2k$ as $k \geq 5$. By Lemma \ref{lemma1}, $G_{k,2}$ has a Hamilton cycle.
\end{proof}

Next we prove the case when $k \geq 3,n \geq 3$ by induction on $n$.  We assume $$\sigma(G_{k,n})\leq\left\{\begin{array}{ll}
  \left(k-\frac{2}{k+1}\right) n & \text { if } k \text { is odd, }  \\
  \left(k-\frac{4}{k+2}\right) n & \text { if } k \text { is even. }
 \end{array}\right.$$ Otherwise, by Theorem \ref{theorem3}, $G_{k,n}$ is Hamiltonian. Without loss of generality, we let $u_1 \in V_1$ and $v \in V_2$ be two nonadjacent vertices such that

 \begin{equation}\label{eq3}
 d_{G_{k,n}}(u_1)+d_{G_{k,n}}(v) \leq\left\{\begin{array}{ll}
  \left(k-\frac{2}{k+1}\right) n & \text { if } k \text { is odd, }  \\
  \left(k-\frac{4}{k+2}\right) n & \text { if } k \text { is even, }
 \end{array}\right.
\end{equation}
 and
  \begin{equation}\label{eq4}
 d_{G_{k,n}}(u_1) \leq\left\{\begin{array}{ll}
  \left(\frac{k}{2}-\frac{1}{k+1}\right) n & \text { if } k \text { is odd, }  \\
  \left(\frac{k}{2}-\frac{2}{k+2}\right) n & \text { if } k \text { is even. }
 \end{array}\right.
\end{equation}
Furthermore, since $d_{\overline{G_{k,n}}}(u)+d_{G_{k,n}}(u) =(k-1)n$ for any vertex $u\in V(G_{k,n})$,  we obtain:
 \begin{equation}\label{eq55}
 d_{\overline{G_{k,n}}}(u_1)+d_{\overline{G_{k,n}}}(v) \geq\left\{\begin{array}{ll}
  \left(k+\frac{2}{k+1}-2\right) n & \text { if } k \text { is odd, }  \\
  \left(k+\frac{4}{k+2}-2\right) n & \text { if } k \text { is even, }
 \end{array}\right.
\end{equation}
 and
  \begin{equation}\label{eq66}
 d_{\overline{G_{k,n}}}(u_1) \geq\left\{\begin{array}{ll}
  \left(\frac{k}{2}+\frac{1}{k+1}-1\right) n & \text { if } k \text { is odd, }  \\
  \left(\frac{k}{2}+\frac{2}{k+2}-1\right) n & \text { if } k \text { is even. }
 \end{array}\right.
\end{equation}

Clearly $\delta(G_{k,n}) \geq 2$ as $|E(G_{k,n}) | \geq \binom{k}{2}n^{2}-(k-1) n+2 $. Therefore $d_{G_{k,n}}(u_1) \geq2$.  We distinguish the following two cases:

{\noindent \bf Case 1.} There exist $i \neq j \in \{2,3,\cdots,k\}$ such that $N_{G_{k,n}}(u_1) \cap V_{i} \neq \emptyset$ and $N_{G_{k,n}}(u_1) \cap V_{j} \neq \emptyset$.

Let $G=G_{k,n}-\{u_1,v\}$. We claim that
\begin{align}  \label{eq5}
\delta(G) \geq (k-2)n.
\end{align}
Otherwise, there exists a vertex $u \in V(G)$ such that $d_G(u) \leq (k-2)n-1$. If $u\in V_1$ or $V_2$, then
\begin{align*}
d_{\overline{G}}(u)  \geq  (k-2)n+(n-1)-((k-2)n-1)=n.
\end{align*}
If $u\in V_i$, $3 \leq i \leq k$, then
\begin{align*}
d_{\overline{G}}(u)  \geq  (k-3)n+2(n-1)-((k-2)n-1)=n-1.
\end{align*}
Combining with (\ref{eq55}), we obtain that
 \begin{equation*}
 |E(\overline{G_{k,n}})|\geq n-1+ \left\{\begin{array}{ll}
 \left(k+\frac{2}{k+1}-2\right) n-1 & \text { if } k \text { is odd, }  \\
  \left(k+\frac{4}{k+2}-2\right) n-1 & \text { if } k \text { is even. }
 \end{array}\right.
\end{equation*}
If $k$ is odd, then
 \begin{align*}
 &n-1+\left(k+\frac{2}{k+1}-2\right) n-1 = (k-1)n-2+ \frac{2}{k+1}n > (k-1)n-2;
\end{align*}
if $k$ is even, then
\begin{align*}
 &n-1+\left(k+\frac{4}{k+2}-2\right) n-1 = (k-1)n-2+ \frac{4}{k+2}n > (k-1)n-2.
\end{align*}
It is a contradiction.

If $N_{G_{k,n}}(u_1) \cap V_2 \neq \emptyset$, without loss of generality, we let $u_2 \in V_2$ and $u_3\in V_3$ be two adjacent vertices to $u_1$.  By \ref{eq5}, we can greedily find a path $P=u_2-u_1-u_3-\cdots-u_k$ of length $k$, where $u_i\in V_i$ for $i=2,\cdots,k$. Otherwise there exists $3 \leq i \leq k-1$ such that $N_G(u_i) \cap V_{i+1}=\emptyset$. Then  $d_G(u_i) \leq 2(n-1)+(k-4)n=(k-2)n-2$, a contradiction. If $N_{G_{k,n}}(u_1) \cap V_2 = \emptyset$, without loss of generality, we let $u_3 \in V_3$ and $u_4\in V_4$ be two adjacent vertices to $u_1$. Similarly, by \ref{eq5}, we can greedily find a path $P=u_3-u_1-u_4-\cdots-u_k-u_2$, where $u_i\in V_i$ for $i=2,\cdots,k$ and $u_2 \neq v$. Otherwise there exists $4 \leq i \leq k-1$ such that $N_G(u_i) \cap V_{i+1}=\emptyset$ or $N_G(u_k) \cap V_{2}\setminus\{v\}=\emptyset$. Then  $d_G(u_i) \leq  (k-2)n-1$, a contradiction. Clearly the path $P$ contains one and only one vertex from every part.

Suppose $P=u_2-u_1-u_3-u_4-\cdots-u_k$ (The case when $P=u_3-u_1-u_4-\cdots-u_k-u_2$ is similar).  By  (\ref{eq111}) and (\ref{eq66}), we obtain that $\overline{G_{k,n}-V(P)}$ is a $(n-1)-$balanced $k$-partite graph with at most $\left(\frac{k}{2}-\frac{1}{k+1}\right) n-2$ edges if $k$ is odd and at most $\left(\frac{k}{2}-\frac{2}{k+2}\right) n-2$ edges if $k$ is even. It is not difficult to check that $$\max\Big\{\left(\frac{k}{2}-\frac{1}{k+1}\right) n-2, \left(\frac{k}{2}-\frac{2}{k+2}\right) n-2 \Big\} \leq (k-1)(n-1)-2$$
as  $k \geq 3$, $n \geq 3$. It follows that $G_{k,n}-V(P)$ contains at least $\binom{k}{2}(n-1)^{2}-(k-1)(n-1)+2$ edges. By inductive hypothesis, $G_{k,n}-V(P)$ contains a Hamilton cycle, denoted by $$C = v_1-v_2-\cdots-v_{k(n-1)}-v_1.$$
If $k(n-1)$ is odd, we construct a matching $M$ of size $\frac{k(n-1)-1}{2}$ from the edges of $C$ with $v \notin V(M)$:  $$M=\Big\{\{v_{2i-1},v_{2i}\}: i=1,2,\cdots,\frac{k(n-1)-1}{2}\Big\}.$$
If $k(n-1)$ is even,  we construct a matching $M$ of size $\frac{k(n-1)-2}{2}$ from the edges of $C$ with $v \notin V(M)$:  $$M=\Big\{\{v_{2i-1},v_{2i}\}: i=1,2,\cdots,\frac{k(n-1)-2}{2}\Big\}.$$
We have the following claim.
\begin{claim}\label{claim777}
There exists one edge $\{v_{2i-1},v_{2i}\}$ of $M$ such that $\{u_2,v_{2i-1}\} \in E(G_{k,n})$ and $\{u_k,v_{2i}\}\in E(G_{k,n})$ or $\{u_2,v_{2i}\} \in E(G_{k,n})$ and $\{u_k,v_{2i-1}\}\in E(G_{k,n})$.
\end{claim}
\begin{proof}  To the contrary, if $k(n-1)$ is odd,  the number of vertices in $\overline{G_{k,n}-V(P)-v}$, which is adjacent to $u_2$ or $u_k$,  is at least  $$ 2\frac{k(n-1)-1}{2} - 2(n-1)=kn-k-2n+1,$$
as the number of vertices from the same part to $u_2$ ($u_k$) is at most $(n-1)$ in $\overline{G_{k,n}-V(P)-v}$. By (\ref{eq55}), we obtain:
\begin{equation*} |E(\overline{G_{k,n}})|  \geq kn-k-2n+1+ \left\{\begin{array}{ll}
\left(k+\frac{2}{k+1}-2\right) n-1 & \text { if } k \text { is odd, }  \\
\left(k+\frac{4}{k+2}-2\right) n-1 & \text { if } k \text { is even. }
 \end{array}\right\}
\end{equation*}
If $k$ is odd, then
\begin{align*}
 & kn-k-2n+1+ \left(k+\frac{2}{k+1}-2\right)n-1 = (k-1)n-2+ \left(k-3+\frac{2}{k+1}\right) n-k+2 \\
\geq & (k-1)n-2+ 3\left(k-3+\frac{2}{k+1}\right)-k+2=(k-1)n-2+ 2k+\frac{6}{k+1} -7>  (k-1)n-2.
\end{align*}
We derive that $|E(\overline{G_{k,n}})| > (k-1)n+2$. It is similar when $k$ is even. It is a contradiction.
\end{proof}

By Claim \ref{claim777}, we obtain a Hamilton cycle of $G_{k,n}$ from $P$ and $C$. Without loss of generality, we can assume $\{v_{1},v_{2}\}$ of $M$ satisfying: $\{u_2,v_1\},\{u_k,v_{2}\}\ \in E(G_{k,n})$. Obviously $v_1-u_2-P-u_k-v_2-C-v_1$ is a Hamilton cycle of $G_{k,n}$.

{\noindent \bf Case 2.}  There exists only one $i \in \{2,3,\cdots,k\}$ such that $N_{G_{k,n}}(u_1) \cap V_{i} \neq \emptyset$.

In this case we have $d_{G_{k,n}}(u_1) \leq n$.  If $k=3$ and $n=3$, then it is easy to check that $G_{3,3}$ contains a Hamilton cycle.  Suppose that $k\geq3$, $n\geq3$ except $k=3$ and $n=3$.  We have the following claim.
\begin{claim}\label{claim8}
 $\delta(G_{k,n}-u_1) \geq (k-2)n+1$.
\end{claim}
\begin{proof} Let $G' =G_{k,n}-u_1$. To the contrary, suppose $u \in V(G')$ such that $d_{G'}(u) \leq (k-2)n$.  If $u\in V_1$,  then $|E(G'-u)|\leq
\binom{k-1}{2}n^2+(n-2)(k-1)n$. It follows that
\begin{align*}
|E(G_{k,n})| & \leq  d_{G'}(u)+ |E(G'-u)|+d_{G_{k,n}}(u_1) \\  & \leq (k-2)n+
\binom{k-1}{2}n^2+(n-2)(k-1)n+n= \binom{k}{2}n^2-(k-1)n,
\end{align*}
a contradiction. We assume that $u \in V_i$ for some  $2 \leq i \leq k$.  Then $|E(G'-u)|\leq
\binom{k-2}{2}n^2+(2n-2)(k-2)n+(n-1)(n-1)$.  Therefore
\begin{align*}
|E(G_{k,n})| & \leq  d_{G'}(u)+ |E(G'-u)|+d_{G_{k,n}}(u_1) \\  & \leq (k-2)n+\binom{k-2}{2}n^2+(2n-2)(k-2)n+(n-1)(n-1)+n= \binom{k}{2}n^2-(k-1)n+1,
\end{align*}
a contradiction.  \end{proof}

Suppose that $N_{G_{k,n}}(u_1) \subseteq V_2$. Let $u_2,u_2' \in N_{G_{k,n}}(u_1) \subseteq V_2$.  We claim that there exist two disjoint paths of length $k$:
\begin{align*}
P_1=u_1-u_2-u_3-\cdots-u_k {  \,\, \text {and} \,\,} P_2=u_1-u_2'-u_3'-\cdots-u_k'-u_1',
\end{align*}
where $u_i \neq u_i'\in V_i$ for $i=1,2,\cdots,k$. First, we can greedily construct $P_1$.  By Claim \ref{claim8},  $\delta(G_{k,n}-u_1) \geq (k-2)n+1$, we obtain that  $N_{G_{k,n}-u_1}(u_i) \cap V_{i+1} \neq \emptyset$ for $2 \leq i \leq k-1$. Next,  we can also greedily construct  $P_2$. Otherwise, there exists $u_i' \in V_i$ such that $N_{G_{k,n}-u_1}(u_i') \cap V_{i+1}\setminus\{u_i\} = \emptyset$ for $2 \leq i \leq k-1$ or $N_{G_{k,n}-u_1}(u_k') \cap V_{1}\setminus\{u_1\} = \emptyset$. It follows that $d_{G_{k,n}-u_1}(u_i') \leq kn-1-n-(n-1)=(k-2)n$, a contradiction.

Suppose that $N_{G_{k,n}}(u_1) \subseteq V_i$ for $3 \leq i \leq k$. Without loss of generality, we let $u_3,u_3' \in N_{G_{k,n}}(u_1) \subseteq V_3$.  We claim that there exist two disjoint paths of length $k$:
\begin{align*}
P_1=u_1-u_3-u_2-u_4-\cdots-u_k {  \,\, \text {and} \,\,} P_2=u_1-u_3'-u_2'-u_4'-\cdots-u_k'-u_1',
\end{align*}
where $u_i \neq u_i'\in V_i$ for $i=1,2,\cdots,k$. Especially,  if $k=3$, we can let $u_2\neq v$.  The explanation is similar to the case when $N_{G_{k,n}}(u_1) \subseteq V_2$.

Recall that $d_{G_{k,n}}(u_1) \leq n$.  Clearly  $d_{\overline{G_{k,n}}}(u_1) \geq (k-2)n$. By (\ref{eq111}), we derive  that $$|E(\overline{G_{k,n}-V(P_1\cup P_2)})| \leq n-2 .$$
Therefore $G_{k,n}-V(P_1)-V(P_2)$ is a $(n-2)-$balanced $k$-partite graph with at least
\begin{align*}
\binom{k}{2}(n-2)^{2}-(n-2) \geq \binom{k}{2}(n-2)^{2}-(k-1)(n-2)+2
\end{align*}
edges as $k\geq3$, $n\geq3$ except $k=3$, $n=3$. By inductive hypothesis, $G_{k,n}-V(P_1 \cup P_2)$ contains a Hamilton cycle, denoted by $$C = v_1-v_2-\cdots-v_{k(n-2)}-v_1.$$
If $k(n-2)$ is odd, we construct a matching $M$ of size $\frac{k(n-2)-1}{2}$ from the edges of $C$ with $v \notin V(M)$:  $$M=\Big\{\{v_{2i-1},v_{2i}\}: i=1,2,\cdots,\frac{k(n-2)-1}{2}\Big\}.$$
If $k(n-2)$ is even,  we construct a matching $M$ of size $\frac{k(n-2)-2}{2}$ from the edges of $C$ with $v \notin V(M)$:  $$M=\Big\{\{v_{2i-1},v_{2i}\}: i=1,2,\cdots,\frac{k(n-2)-2}{2}\Big\}.$$
We have the following claim.
\begin{claim}\label{claim7}
There exists one edge $\{v_{2i-1},v_{2i}\}$ of $M$ such that $\{u_1',v_{2i-1}\} \in E(G_{k,n})$ and $\{u_k,v_{2i}\}\in E(G_{k,n})$ or $\{u_1',v_{2i}\} \in E(G_{k,n})$ and $\{u_k,v_{2i-1}\}\in E(G_{k,n})$.
\end{claim}
\begin{proof} To the contrary, if $k(n-2)$ is odd,  the number of vertices in $\overline{G_{k,n}-V(P_1\cup P_2)-v}$, which is adjacent to $u_1'$ or $u_k$,  is at least  $$ 2\frac{k(n-2)-1}{2} - 2(n-2)=kn-2k-2n+3,$$
as the number of vertices from the same part to $u_1'$ ($u_k$) is at most $(n-2)$ in $\overline{G_{k,n}-V(P_1\cup P_2)-v}$. By (\ref{eq55}), we obtain:
\begin{equation*} |E(\overline{G_{k,n}})|  \geq kn-2k-2n+3+ \left\{\begin{array}{ll}
\left(k+\frac{2}{k+1}-2\right) n-1 & \text { if } k \text { is odd, }  \\
\left(k+\frac{4}{k+2}-2\right) n-1 & \text { if } k \text { is even. }
 \end{array}\right\}
\end{equation*}
If $k$ is odd, then
\begin{align*}
 & kn-2k-2n+3+ \left(k+\frac{2}{k+1}-2\right)n-1 \\
= &  (k-1)n-2+  \left(k-3+\frac{2}{k+1}\right) n-2k+4 \\
\geq & (k-1)n-2+ 3\left(k-3+\frac{2}{k+1}\right)-2k+4\\
= & (k-1)n-2+ k+\frac{6}{k+1}-5 >(k-1)n-2.
\end{align*}
We derive that $|E(\overline{G_{k,n}})| > (k-1)n+2$. It is similar when $k$ is even. It is a contradiction.
\end{proof}

By Claim \ref{claim7}, we obtain a Hamilton cycle of $G_{k,n}$ from $P_1$, $P_2$ and $C$. Without loss of generality, we can assume $\{v_{1},v_{2}\}$ of $M$ satisfying: $\{u_1',v_1\},\{u_k,v_{2}\}\ \in E(G_{k,n})$. Obviously $v_1-u_1'-P_2-u_1-P_1-u_k-v_2-C-v_1$ is a Hamilton cycle of $G_{k,n}$.

Once we have proved Theorem \ref{theorem1}, we can obtain a useful inference.
\begin{theorem}\label{theorem11} Let \normalsize $G=(V,E)$ be an $n$-balanced $k$-partite graph with $k \geq 2$,  $n \geq 1$ except $k=2$, $n=1$.  If $|E(G)| \geq n^{2} C_{k}^{2}-(k-1) n+1 $ and $\delta(G_{k,n}) \geq 2$, then $G$ is Hamiltonian.
\end{theorem}

\section{ Proof of Theorem \ref{theorem11}}

Suppose that  $G_{k,n}=(V(G_{k,n}),E(G_{k,n}))$ is an $n$-balanced $k$-partite graph with $k$ parts  $V_1$, $V_2$, $\cdots,$ $V_k$ , $|E(G_{k,n})| \geq \binom{k}{2}n^{2}-(k-1) n+1 $ and $\delta(G_{k,n}) \geq 2$, where $k \geq 2$,  $n \geq 1$ except $k=2$, $n=1$. Recall that $E(\overline{G_{k,n}})=E(CG_{k,n})\setminus E(G_{k,n})$,  we have
\begin{equation}\label{eq121}
 |E(\overline{G_{k,n}})| \leq(k-1) n-1.
\end{equation}

For every pair of nonadjacent vertices $u$ and $v$ which are in different partite sets,$$
d(u)+d(v) \geq \max \{4,(k-1) n\}=\sigma_{G_{k,n}}
$$

If $G$ satisfies the conditions of Theorem\ref{theorem3}, $G$ is hamiltonian. Otherwise, there exists a pair of nonadjacent vertices $u$ and $v$ which are in different partite sets such that
\begin{equation}
d_{\overline{G_{k,n}}}(u) \leq\left\{\begin{array}{l}
\left(\frac{k}{2}-\frac{1}{k+1}\right) n-1 \text { if } k \text { is odd } \\
\left(\frac{k}{2}-\frac{2}{k+2}\right) n-1 \text { if } k \text { is even }
\end{array}\right.
\end{equation}
\begin{equation}\label{eq1111}
\left|E(\overline{G_{k,n} - \{u\}}) \bigcap E(\overline{G_{k,n} - \{v\}})\right| \leq\left\{\begin{array}{l}
\left(1-\frac{2}{k+1}\right) n \text { if } k \text { is odd } \\
\left(1-\frac{4}{k+2}\right) n \text { if } k \text { is even }
\end{array}\right.
\end{equation}
hold.

$\delta(G_{k,n}) + 1>\left\{\begin{array}{l}\left(k-\frac{2}{k+1}\right) n \quad \text{if k is odd}\\ \left(k-\frac{4}{k+2}\right) n \quad \text{if k is even}\end{array}\right.$, only when $k=2,n \geq 2$, or $k=4,n=2$, or $k \geq 2,n=1$. In these several cases, we only need to consider $d(u)+d(v)=\delta(G_{k,n})>\left\{\begin{array}{l}\left(k-\frac{2}{k+1}\right) n \quad \text{if k is odd}\\ \left(k-\frac{4}{k+2}\right) n \quad \text{if k is even}\end{array}\right.$. It’s straightforward to prove these cases with a comparable proof of the $k \geq 2,n=2$ case in Proof of Claim \ref{claim88}.

The other case is $\delta(G_{k,n}) + 1 \leq \left\{\begin{array}{l}\left(k-\frac{2}{k+1}\right) n \quad \text{if k is odd}\\ \left(k-\frac{4}{k+2}\right) n \quad \text{if k is even}\end{array}\right.$. If $d(u)+d(v)=\delta(G_{k,n})$, it’s straightforward to prove these cases with a comparable proof of the  $k \geq 2,n=2$ case in Proof of Claim \ref{claim88}. If $d(u)+d(v) > \delta(G_{k,n})$, then $|E(\overline{G_{k,n} - \{u\}}) \bigcap E(\overline{G_{k,n} - \{v\}})| \geq 1$. There must exist an edge belonging to $E(\overline{G_{k,n}})$ that is not associated with $u$ or $v$. It is useful to set the edge to $ab$. Considering the graph $G \bigcup\{a b\}$, set it to $G$. Applying Theorem 1\ref{theorem1}, $G$ is Hamiltonian. Let the Hamiltonian cycle of $G^{\prime}$ be $H$. If $H$ does not include the edge $ab$, the conclusion holds. Otherwise, $H$ contains the edge $ab$. Let the other points adjacent to $a,b$ on $H$, respectively, be $a_{left},b_{right}$. It’s useful to set $H: w \ldots a_{l e f t} a b b_{r i g h t} \ldots w$. Because $d(a) \geq 2$ and $d(b) \geq 2$ in $G$, let $a^{\prime} \in N(a), b^{\prime} \in N(b), a^{\prime} \neq a_{l e f t}, b^{\prime} \neq b_{r i g h t}$. If $H$ contains the edge $a^{\prime}b^{\prime}$, there exists a new Hamiltonian cycle $H^{\prime}: w-\overrightarrow{H}-a^{\prime} a-\overleftarrow{H}-b^{\prime} b-\overrightarrow{H}-w\left(H^{\prime}: w-\overrightarrow{H}-a a^{\prime}-\overleftarrow{H}-b^{\prime} b-\overrightarrow{H}-w\right)$
And the conclusion holds. Thus, $|N_{\overline{G_{k,n}}}(a) \bigcup N_{\overline{G_{k,n}}}(b)| \geq k n-3+1-(n-1)-(n-1)=(k-2) n$, then $|E(\overline{G_{k,n} - \{u\}}) \bigcap E(\overline{G_{k,n} - \{v\}})| \geq |N_{\overline{G_{k,n}}}(a) \bigcup N_{\overline{G_{k,n}}}(b)| -4 \geq(k-2) n-4$. Since (\ref{eq1111}), the contradiction arises if $k$ and $n$ are sufficiently large.

\bibliographystyle{plain}
\bibliography{achemso-demo}
\end{document}